\documentclass[11pt,reqno]{amsart}
\usepackage{amsfonts, amsmath, amssymb, amscd, amsthm, bm}
\usepackage{enumerate}
\usepackage{url}
\usepackage[
linktocpage=true,colorlinks,citecolor=magenta,linkcolor=blue,urlcolor=magenta]{hyperref}
\usepackage{multicol}
\usepackage{cite}
\usepackage{comment}
\usepackage{graphicx}
\usepackage{xcolor}
\usepackage{verbatim}
\usepackage{textcomp}
\usepackage{latexsym}
\usepackage{mathrsfs}

\usepackage[margin=1.1in]{geometry}
 \newtheorem{theorem}{Theorem}[section]

 \newtheorem{lemma}[theorem]{Lemma}
 
 \newtheorem{proposition}[theorem]{Proposition}

 \theoremstyle{definition}
 
  \newtheorem*{ack}{Acknowledgments}

 \theoremstyle{remark}
 \newtheorem{remark}{Remark}[section]

\numberwithin{equation}{section}
\allowdisplaybreaks

\begin{document}
\title{Geometric inequalities involving three quantities in warped product manifolds}
\author[K.-K. Kwong]{Kwok-Kun Kwong}
\address{Mathematical Sciences Institute, Australian National University, Canberra, ACT 2601, Australia}
\email{\href{mailto:kwok-kun.kwong@anu.edu.au}{kwok-kun.kwong@anu.edu.au}}

\author[Y. Wei]{Yong Wei}
\address{School of Mathematical Sciences, University of Science and Technology of China, Hefei 230026, P.R. China}
\email{\href{mailto:yongwei@ustc.edu.cn}{yongwei@ustc.edu.cn}}
%

\begin{abstract}
In this paper, we establish two families of sharp geometric inequalities for closed hypersurfaces in space forms or other warped product manifolds. Both families of inequalities compare three distinct geometric quantities. The first family concerns the $k$-th boundary momentum, area, and weighted volume, and has applications to Weinstock-type inequalities for Steklov or Wentzell eigenvalues on star-shaped mean convex domains. This generalizes the main results of \cite{Weins21}. The second family involves a weighted $k$-th mean curvature integral and two distinct quermassintegrals and extends the authors' recent work \cite{kwong2022inverse} with G. Wheeler and V.-M. Wheeler.
\end{abstract}

\maketitle
\tableofcontents

\section{Introduction}

In the first part of this paper, we investigate weighted geometric inequalities of the form (or a slight variation of it)
\begin{equation}\label{ineq eg1}
\int_{\Sigma} \lambda^k d \mu \geq \frac{n}{n+k} \omega_n^{-\frac{k}{n}}|\Sigma|^{\frac{n+k}{n}}+k \int_{\Omega} \lambda^{k-1} \lambda^{\prime} d v,
\end{equation}
where $\Sigma=\partial \Omega$ is the boundary of a smooth domain and $\omega_n=|\mathbb S^n|$ is the area of unit sphere $\mathbb{S}^n$. We consider these inequalities in a broad class of warped product manifolds $\overline{M}^{n+1}=[a, b) \times N^n$, endowed with the Riemannian metric $\overline{g}=d r^2+\lambda(r)^2 g_N$, where $g_N$ is a Riemannian metric on the base manifold $N$. The precise assumptions on $\overline {M}^{n+1}$ and $\Sigma$ will be specified in the theorems that follow.

Let us give some motivation for considering weighted integrals such as $\int_{\Sigma} \lambda(r)^k d \mu$ and $\int_{\Omega} \lambda^{k-1} \lambda^{\prime} d v$. To better understand their significance, we consider the case in $\mathbb R^{n+1}$. In this case, these integrals correspond to the weighted perimeter and volume, where the weights are powers of the distance $r=|x|$ from the origin. These weighted inequalities are commonly known as weighted isoperimetric problems in $\mathbb R^{n+1}$, or more broadly, as isoperimetric problems in manifolds with density \cite{rosales2008isoperimetric}. See for instance \cite{betta1999weighted, betta2008weighted, cabre2013Sobolev}. These problems have been the focus of extensive research over the last few decades. One particularly interesting example of these problems is the isoperimetric problem with Gaussian density $\exp \left(-c r^2\right)$ \cite{borell1975brunn}, which has applications in probability and statistics. Another noteworthy example is the isoperimetric problem with radial density, where the weight is some power of the distance $r$. In the case where the power is two, the weighted integral is referred to as the polar moment of inertia, an important quantity in Newtonian physics. Mathematically, there has also been extensive investigation into weighted Caffarelli-Kohn-Nirenberg interpolation inequalities \cite{caffarelli1984first, catrina2001caffarelli}, weighted Sobolev inequalities \cite{cabre2013Sobolev} and weighted isoperimetric inequalities \cite{alvino2017some,canete2010some}, where the integrals are weighted by a power of $r$.

Often, these Sobolev or isoperimetric inequalities compare two weighted integrals or norms. For example, in \cite{alvino2017some}, the problem of minimizing $\int_{\partial \Omega}r^k d\mu$ subject to the constraint $\int_{\Omega}r^l d v=1$ for smooth domain $\Omega\subset \mathbb R^{n+1}$ is considered, resulting in some isoperimetric inequalities of type $\int_{\partial \Omega}r^k d\mu\ge C_{k, l, n} \left(\int_{\Omega}r^l d v\right)^{\frac{ k+n }{ l+n+1 }}$ for a range of $k, l$ and $n$. Another similar but much simpler inequality is
\begin{equation}\label{rv}
\int_{\Sigma} r d \mu\ge (n+1) |\Omega|,
\end{equation}
which is easily obtained by the divergence theorem. In \cite{girao2020weighted}, Gir\~{a}o and Rodrigues showed that although the conjectured inequality $\int_{\Sigma} r d \mu \ge \omega_{n}\left(\frac{|\Sigma|}{\omega_{n}}\right)^{\frac{n+1}{n}}$ is false even for star-shaped mean convex surfaces (which, if true, would improve \eqref{rv} by the isoperimetric inequality), the inequality \eqref{rv} can still be improved by the following weighted geometric inequality involving three terms, when the hypersurface $\Sigma$ is
star-shaped and strictly mean convex:
\begin{equation}\label{s1.girao}
\int_{\Sigma} r d \mu\ge \frac{n}{n+1} \omega_{n}\left(\frac{|\Sigma|}{\omega_{n}}\right)^{\frac{n+1}{n}}+ |\Omega|.
\end{equation}
This is an improvement due to the isoperimetric inequality, and is a special case of our family of inequalities \eqref{ineq eg1} when $k=1$ and $\lambda(r)=r$. It turns out that a number of weighted geometric inequalities comparing two quantities can be improved by introducing another term (such as $|\Sigma|^\alpha$ for some $\alpha$) into the inequalities, but at the same time it cannot be improved by dropping any one of these terms. This is one of the reasons why we consider three-term inequalities such as \eqref{ineq eg1} in this paper.

We would like to give another motivation for considering inequalities involving three quantities, and why it can be a challenging problem. In \cite{Weins21}, the following isoperimetric inequality is proved for convex set $\Omega\subset \mathbb R^{n+1}$:
\begin{equation}\label{iso.intro}
\int_{\partial \Omega}r^2 d \mu \geq b_{n+1}^{-\frac{2}{n+1}}|\partial \Omega \| \Omega|^{\frac{2}{n+1}}.
\end{equation}
Here $b_{n+1}=|\mathbb B^{n+1}|$ is the volume of the unit ball in $\mathbb{R}^{n+1}$. This in turn is used to prove some Weinstock type inequalities (\eqref{s2.weins1}, \eqref{wentzell}) for the first nonzero Steklov and Wenzell eigenvalue for convex domains in $\mathbb{R}^{n+1}$, generalizing the inequality \eqref{weinstock} of Weinstock \cite{Weins54,Weins54-2}. The inequality \eqref{iso.intro} is not true for arbitrary domain (e.g. for a smooth domain near the origin, with a lot of boundary area).
Among sets with prescribed volume, both the boundary momentum $\int_{\partial \Omega}r^2d\mu$ and the surface area $|\partial \Omega|$ are minimized by balls. Therefore, the competition of these two terms poses a challenge. The inequality \eqref{iso.intro} states the minimum of the ratio of these two quantities is minimized by balls, which is far from obvious. It is worth noting that the inequality \eqref{iso.intro} can be deduced from \eqref{s1.girao} using elementary methods. See Section \ref{sec weinstock} for the details. This not only extends the results of \cite{Weins21} to a more general class of domains, but also provides a simpler method, in our opinion.

The main ingredient in proving the kind of inequalities such as \eqref{ineq eg1} is the inverse mean curvature flow (IMCF).
There is already a number of geometric inequalities which can be obtained by finding a monotone quantities along the IMCF and investigating the asymptotic of the (rescaled) limit. However, most of the inequalities obtained this way consist of only two geometric quantities. There are some exceptions, see for example \cite{BHT2016, girao2020weighted, kwong2022inverse}, but the results are rather isolated. An idea to obtain a three-term inequality is to obtain a monotone quantity $Q(t)$ along IMCF consisting of two geometric quantities. In Section \ref{sec.monot}, we are going to find a family of such quantities $Q(t)=Q_k(t)$ for each $k\in \mathbb R$. This will enable us to systematically find the one-parameter family \eqref{ineq eg1} of inequalities involving three geometric quantities, and the proof will be given in Section \ref{sec ineq k boundary moment}.

In the second part of this paper, we investigate three-term geometric inequalities involving weighted curvature integrals and quermassintegrals. The quermassintegrals $W_k(\Omega)$ are fundamental objects in convex geometry and play a central role in the study of convex sets. Section \ref{sec quermassintegrals} will provide the precise definition of quermassintegrals, but to give some examples, we have $W_0(\Omega)=|\Omega|$, $W_1(\Omega)=\frac{1}{n}|\partial \Omega|$, and $W_{n+1}(\Omega)=\left|\mathbb{B}^{n+1}(1)\right|$ for smooth domains in space forms. The literature contains numerous inequalities that compare various quermassintegrals, such as the Brunn-Minkowski and Alexandrov-Fenchel inequalities, cf. \cite{Schneider2014} and references therein. Historically, these inequalities were first proved using various analytic methods. However, it has been relatively recent that mathematicians have found a way to prove these inequalities using geometric flows, particularly inverse curvature type flows, at least when $\Omega$ is smooth enough. See for example \cite{ACW21,AHL21,BGL,BHT2016,GeWW14,GL09,GL15,delima16,Chen-Sun22,Wx14} and the references therein. As explained in the first part, the key to proving these inequalities is often finding a monotone quantity along the flow and investigating its asymptotic behavior. However, inequalities proved using this method typically only compare two geometric quantities. This is because the flow is often designed to keep one quantity constant while making the other quantity monotone along the flow. Deriving inequalities that involve more than two quermassintegrals using geometric flow is often challenging. For instance, to the best of our knowledge, it is unknown whether the inequality $W_k(\Omega)^2\ge c_{k,n} W_{k-1}(\Omega) W_{k+1}(\Omega)$ \cite[(7.66)]{Schneider2014} for smooth domains in $\mathbb R^{n+1}$ with enough convexity can be obtained using geometric flow methods.

Section \ref{sec quermassintegrals} is an attempt to remedy this situation by constructing a family of three-term geometric inequalities involving weighted curvature integrals and quermassintegrals. This idea was inspired by the authors' recent work \cite{kwong2022inverse} with G. Wheeler and V.-M. Wheeler, where they proved that for a smooth, closed and convex curve $\gamma$ in $\overline{M}^2=\mathbb{H}^2, \mathbb{S}_{+}^2$ (open hemisphere) or $\mathbb{R}^2$, it holds that
\begin{equation}\label{kwww}
\int_\gamma \Phi(r) \kappa d s \geq \frac{1}{2 \pi}\left(L^2-2 \pi A\right)
\end{equation}
where
$
\Phi(r)=
\begin{cases}
1-\cos r, & K=1\\
\cosh r-1, & K=-1 \\
\frac{r^2}{2}, & K=0
\end{cases}
$ and $K$ is the ambient curvature. Here $L$ is the length of $\gamma=\partial \Omega$ and $A$ is the area enclosed by it. This inequality provides a better comparison than the two-term inequalities presented in \cite[Theorem 1.6]{girao2020weighted} and \cite[Theorem 2]{kwong2014new}, where $\int_\gamma \Phi(r) \kappa d s \ge \frac{L^2}{4 \pi}$ and $\int_\gamma \Phi(r) \kappa d s \ge A$, respectively. The function $\Phi$ appears naturally in this inequality due to the conformal symmetry of the ambient space $\overline{M}^2$. Specifically, the gradient of $\Phi$ is a conformal Killing vector field. This property allows us to derive various geometric inequalities involving both the weighted curvature integrals and quermassintegrals. This result turns out to have a number of applications, including a counterexample to the $n=2$ case of a conjecture of Gir\~{a}o-Pinheiro \cite{girao17}. Notice that \eqref{kwww} is a three-term inequality involving $W_0(\Omega)$, $W_1(\Omega)$ and the weighted curvature integral $\int_\gamma \Phi(r) \kappa d s$.

In Section \ref{sec quermassintegrals}, we extend this result to higher dimensional space forms under various convexity assumptions of the domain. For technical reasons, we utilize three slightly different but closely related inverse curvature type flows to derive the result in the Euclidean, hyperbolic, and spherical space forms respectively. In each case, we obtain a family of monotone quantities that involve both the weighted curvature integral $\int_{\partial \Omega_t} \Phi E_k$ and the quermassintegral $W_{k-1}\left(\Omega_t\right)$, where $E_k$ is the normalized $k$-th mean curvature. As an example of the kind of result we prove, consider a smooth, closed, star-shaped, and $k$-convex hypersurface $\Sigma$ in $\mathbb{R}^{n+1}$ with $n \geq 2$, enclosing a bounded domain $\Omega$. Then, for all $k=1, \cdots, n$, the following inequality holds:
\begin{equation}\label{ineq eg2}
\int_{\Sigma} \Phi E_k d \mu+k W_{k-1}(\Omega) \geq \frac{n+2+k}{2(n+2-k)} \omega_n\left(\frac{n+1-k}{\omega_n}\right)^{\frac{n+2-k}{n+1-k}} W_k(\Omega)^{\frac{n+2-k}{n+1-k}},
\end{equation}
where $\Phi(r)=r^2 / 2$. Equality holds if and only if $\Sigma$ is a coordinate sphere.

The organization of this paper is as follows. In Section \ref{sec weinstock}, we present a method for generalizing the main results of \cite{Weins21} to the more general class of star-shaped mean convex domains, that in our opinion is more straightforward. In Section \ref{sec.monot}, we introduce the inverse mean curvature flow and establish the monotonicity of a crucial quantity along this flow. This quantity will then be used in Section \ref{sec ineq k boundary moment} to prove an inequality of the type \eqref{ineq eg1} in various warped product spaces. In Section \ref{sec quermassintegrals}, we begin by introducing the quermassintegrals and the $k$-th mean curvatures, and then proceed to establish an inequality similar to \eqref{ineq eg2} in the three space forms. Our derivation of these inequalities will involve establishing the monotonicity of a geometric quantity that includes $W_{k-1}\left(\Omega_t\right)$ and $\int_{\partial \Omega_t} \Phi E_k$ along an associated inverse curvature type flow.

\begin{ack}
Kwok-Kun Kwong was supported by grant FL150100126 of the Australian Research Council. Yong Wei was surpported by National Key Research and Development Program of China 2021YFA1001800 and 2020YFA0713100, and Research grant KY0010000052 from University of Science and Technology of China.
\end{ack}

\section{Weinstock inequality for star-shaped mean convex domain}\label{sec weinstock}

In this section, we provide a simpler approach for extending the main results of \cite{Weins21} to a broader class of star-shaped mean convex domains.

Let us first revisit the Weinstock inequality for the first nonzero Steklov eigenvalue in $\mathbb{R}^{n+1}$. Let $\Omega\subset \mathbb{R}^{n+1}$ be a bounded domain.
The Steklov eigenvalue problem on $\Omega$ is
$$
\begin{cases}\overline \Delta u=0 & \text { in } \Omega \\ \partial_\nu u=\sigma u & \text { on } \partial \Omega\end{cases}
$$
where $\overline \Delta$ is the Laplacian operator acting on functions on $\Omega$, and $\nu$ is the outward normal along the boundary $\partial \Omega$.

The first non-zero Steklov eigenvalue of $\Omega$ can be characterized by
\begin{equation}\label{s2.stek}
  \sigma(\Omega)=\min\left\{\dfrac{\int_\Omega|\overline \nabla u|^2dv}{\int_{\partial\Omega}u^2d\mu}:\quad u\in H^1(\Omega)\setminus\{0\},~\int_{\partial\Omega}u d\mu=0\right\}.
\end{equation}
Weinstock \cite{Weins54,Weins54-2} proved that if $\Omega\subset \mathbb R^2$ is simply connected, then
\begin{equation}\label{weinstock}
\sigma(\Omega)|\partial\Omega|\leq \sigma(\mathbb{B}^2)|\partial \mathbb{B}^2|,
\end{equation}
where $\mathbb{B}^2\subset \mathbb R^2$ is a ball. For higher dimension, Bucur, Ferone, Nitsch and Trombetti \cite{Weins21} prove that for bounded convex domain $\Omega\subset \mathbb{R}^{n+1}$, there holds
\begin{equation}\label{s2.weins}
\sigma(\Omega)|\partial\Omega|^{\frac{1}{n}}\leq \sigma(\mathbb{B}^{n+1})|\partial \mathbb{B}^{n+1}|^{\frac{1}{n}}.
\end{equation}
Equality holds if and only if $\Omega$ is a ball. The key ingredient in the proof of \eqref{s2.weins} is the following sharp isoperimetric type inequality
\begin{equation}\label{s2.iso}
\int_{\partial\Omega}r^2d\mu\geq {b_{n+1}} ^{-\frac{2}{n+1}}|\partial\Omega||\Omega|^{\frac{2}{n+1}}
\end{equation}
for convex domain $\Omega\subset \mathbb{R}^{n+1}$, where $r^2=|x|^2$ and $b_{n+1}=|\mathbb{B}^{n+1}|$ is the volume of the unit ball in $\mathbb R^{n+1}$. Equality holds in \eqref{s2.iso} if and only if $\Omega$ is a ball centered at the origin.

In \cite{Weins21}, the authors conjectured that \eqref{s2.iso}, and hence the Weinstock inequality \eqref{s2.weins}, holds for star-shaped mean convex domains. In this section, we confirm this conjecture by demonstrating that \eqref{s2.iso} holds for these types of domains. This can be accomplished by using the inequality \eqref{s1.girao} for star-shaped mean convex domain $\Omega\subset \mathbb{R}^{n+1}$, which also establishes that the Weinstock inequality \eqref{s2.weins} applies to such domains.

\begin{theorem}\label{thm weinstock}
Let $\Omega$ be a smooth, bounded domain in $\mathbb{R}^{n+1}$ with star-shaped mean convex boundary $\partial\Omega$. Then the inequality
\begin{equation}\label{s2.weins1}
\sigma(\Omega)|\partial\Omega|^{\frac{1}{n}}\leq \sigma(\mathbb{B}^{n+1})|\partial \mathbb{B}^{n+1}|^{\frac{1}{n}}
\end{equation}
holds, and equality holds if and only if $\Omega$ is a ball.

\end{theorem}
\proof
Firstly, using the H\"{o}lder inequality, we have
\begin{equation}\label{s2.pf1}
\left(\int_{\partial\Omega}rd\mu\right)^2\leq |\partial\Omega|\int_{\partial\Omega}r^2d\mu.
\end{equation}
Secondly, using the Young's inequality, we have
\begin{equation}\label{s2.pf2}
  ((n+1)|\Omega|)^{\frac{1}{n+1}}|\partial\Omega|\omega_n^{-\frac{1}{n+1}}\leq |\Omega|+\frac{n}{n+1}\omega_n^{-\frac{1}{n}}|\partial\Omega|^{\frac{n+1}{n}}
\end{equation}
Combining \eqref{s2.pf1}, \eqref{s2.pf2} and applying \eqref{s1.girao} imply that
\begin{equation}\label{s2.pf3}
  \int_{\partial\Omega}r^2d\mu\geq \left( {\frac{\omega_n}{n+1}} \right)^{-\frac{2}{n+1}}|\partial\Omega||\Omega|^{\frac{2}{n+1}}= {b_{n+1}} ^{-\frac{2}{n+1}}|\partial\Omega||\Omega|^{\frac{2}{n+1}},
\end{equation}
which holds for star-shaped and mean convex domain $\Omega$ in $\mathbb{R}^{n+1}$. Here we used the fact that $b_{n+1}=\frac{\omega_n}{n+1}$.

Applying \eqref{s2.pf3} and following the proof of Theorem 3.1 in \cite{Weins21}, we see that the Weinstock inequality \eqref{s2.weins1} also holds for star-shaped and mean convex domain $\Omega$ in $\mathbb{R}^{n+1}$. We include the proof here for the convenience of readers. We can assume that $\partial\Omega$ has the origin as barycenter. Choosing the coordinate functions $x_i, i=1,\cdots,n+1$ as test functions in \eqref{s2.stek}, we have
\begin{equation*}
  \sigma(\Omega)\int_{\partial\Omega}x_i^2d\mu\leq \int_\Omega |\overline \nabla x_i|^2dv,\quad i=1,\cdots,n+1.
\end{equation*}
Summing up for $i=1,\cdots,n+1$ and noting that $|\overline \nabla x_i|^2=1$, we obtain
\begin{equation}\label{s2.pf4}
  \sigma(\Omega)\int_{\partial\Omega}r^2d\mu\leq (n+1)|\Omega|.
\end{equation}
Substituting \eqref{s2.pf3} into \eqref{s2.pf4} implies that
\begin{equation}\label{s2.weins2}
  \sigma(\Omega)\frac{|\partial \Omega|}{|\Omega|^{\frac{n-1}{n+1}}}\leq b_{n+1}^{\frac{2}{n+1}}=\sigma(\mathbb{B}^{n+1}) \frac{|\partial \mathbb{B}^{n+1}|}{|\mathbb{B}^{n+1}|^{\frac{n-1}{n+1}}}.
\end{equation}
Applying the classical isoperimetric inequality to \eqref{s2.weins2}, we conclude the equality \eqref{s2.weins1}.
\endproof

\begin{remark}
In fact, the proof above provides a stronger inequality \eqref{s2.weins2} for smooth, bounded domain $\Omega$ in $\mathbb{R}^{n+1}$ with star-shaped mean convex boundary $\partial\Omega$. The inequality \eqref{s2.weins1} means that
\begin{equation*}
  \sigma(\Omega)\leq \sigma(\Omega^*)
\end{equation*}
where $\Omega^*$ is a ball such that $|\partial\Omega|=|\partial\Omega^*|$. Equality holds if and only if $\Omega$ is a ball.
\end{remark}

Using a similar idea, the inequality \eqref{s1.girao} can also be applied to study the isoperimetric inequality for the Wentzell eigenvalue, which is a generalization of the Steklov eigenvalue. Recall that the Wentzell eigenvalue problem is the problem
$$
\begin{cases} \overline \Delta u=0 & \text { in } \Omega \\
-\beta \Delta u+\partial_\nu u=\mu u & \text { on } \partial \Omega
\end{cases}
$$
where $\beta$ is a given real number, $\partial_\nu$ denotes the outward unit normal derivative, and $\Delta$ denotes the Laplacian on $\partial \Omega$. For the physical interpretation of this problem, see \cite{dambrine2016extremal}. Here, we only consider the case where $\beta\ge 0$. Similar to the Steklov eigenvalue, the first non-zero Wentzell eigenvalue $\mu(\Omega,\beta)$ satisfies
\begin{equation*}
\begin{aligned}
\mu(\Omega,\beta)=
\min \left\{\frac{\int_{\Omega}|\overline \nabla u|^2 dv+\beta \int_{\partial \Omega}\left|\nabla u\right|^2 d \mu}{\int_{\partial \Omega} u^2 d \mu}: u \in H^1(\Omega) \setminus\{0\}, \int_{\partial \Omega} u d \mu=0\right\},
\end{aligned}
\end{equation*}
where $\nabla$ is the tangential gradient along $\partial \Omega$.

In \cite[Theorem 3.2]{Weins21}, it was shown that among all open convex sets in $\mathbb{R}^n$ with the same volume, the functional $\mu(\Omega, \beta)$ is maximized by a ball. We can extend this result to the class of star-shaped mean convex domains using the same idea as in Theorem \ref{thm weinstock}. Notably, the inequality \eqref{s2.pf3} is again the crucial ingredient in the proof of \cite[Theorem 3.2]{Weins21}, and since convexity is not required elsewhere in the proof, we will omit the details and only state the result. The reader may refer to \cite{Weins21} for more details.
\begin{theorem}
Let $\Omega$ be a bounded, star-shaped mean convex domain of $\mathbb{R}^{n+1}$ and $\beta\ge 0$. Then
\begin{equation}\label{wentzell}
\mu(\Omega, \beta) \leq \mu(\Omega^{\sharp}, \beta),
\end{equation}
where $\Omega^{\sharp}$ is a ball such that $ |\Omega| = |\Omega^{\sharp}|$. The equality holds if and only if $\Omega$ is a ball.
\end{theorem}

\section{A monotone quantity along the IMCF in warped product manifolds}\label{sec.monot}
Assume that $N^n$ is a compact Riemannian manifold with metric $g_N$. We consider the warped product manifold $\overline{M}^{n+1}=[a,b)\times N^n$ equipped with the Riemannian metric
\begin{equation*}
\overline{g}=dr^2+\lambda(r)^2g_{N},
\end{equation*}
where $a\geq 0$, and $b$ is allowed to be $\infty$, $\lambda:[a,b)\to \mathbb{R}$ is a smooth function and is positive on $(a,b)$. We also allow that $\{a\}\times N$ degenerates to a point. As examples, the space forms can be viewed as warped product manifolds
$\overline{M}^{n+1}=I\times \mathbb{S}^n$ for an interval $I$.
\begin{equation}\label{s3.spaceforms}
\overline{M}^{n+1}=\left\{
\begin{aligned}
\mathbb{R}^{n+1},&\qquad \mathrm{if}~I=\mathbb{R}_\geq, \quad \lambda(r)=r, \\
\mathbb{H}^{n+1},&\qquad \mathrm{if}~I=\mathbb{R}_\geq, \quad \lambda(r)=\sinh r,\\
\mathbb{S}^{n+1},&\qquad \mathrm{if}~I=[0,\pi), \quad \lambda(r)=\sin r. \\
\end{aligned}
\right.
\end{equation}

Let $\Sigma$ be a closed, embedded orientable hypersurface in  $\overline{M}^{n+1}$. Then there are two cases: (i). $\Sigma$ is null-homologous and is the boundary of a bounded domain $\Omega$; (ii). $\Sigma $ is homologous to the boundary $\{a\}\times N^n$ and there is a bounded domain $\Omega$ with $\partial\Omega=\Sigma\cup (\{a\}\times N^n)$. To simplify the notation, we let $\Gamma=\emptyset$ in case (i) and let $\Gamma=\{a\}\times N^n$ in case (ii). Then in each case we have $\partial\Omega=\Sigma\cup \Gamma$. Let $\nu$ be the unit outward normal to $\Sigma$, and $\eta$ be the outward normal to $\Gamma$ when it is nonempty. We have
\begin{lemma}\label{s2.lem}
Let $\Sigma$ be a closed, embedded orientable hypersurface in $\overline{M}^{n+1}$. Then
\begin{equation}\label{s2.lemma1}
\int_{\Sigma}\lambda(r)^kd\mu\geq (n+k)\int_\Omega \lambda^{k-1}(r){\lambda}'(r)dv+\lambda^k(a)|\Gamma|,
\end{equation}
where $d\mu$ denotes the area form on $\Sigma$ and $dv$ denotes the volume form on $\overline{M}^{n+1}$.  The equality holds in \eqref{s2.lemma1} if and only if $\Sigma$ is a slice $\{r\}\times N$ for some $r\in (a,b)$.
\end{lemma}
\proof
Recall that the vector field $V=\lambda(r)\partial_r$ is a conformal Killing field (see \cite[Lemma 2.2]{Br13}) and satisfies
\begin{equation*}
\langle \overline \nabla_{e_i}V,e_j\rangle ={\lambda}'(r)\overline{g}_{ij},
\end{equation*}
where ${\lambda}'(r)$ denotes the derivative of $\lambda(r)$ with respect to $r$. Let $Y=\lambda^k(r)\partial_r$. For each $p\in \overline{M}^{n+1}$, choose an orthonormal frame $\{e_1, e_2,\cdots, e_{n+1}\}$ around $p$ such that $e_1=\partial_r$. Then
\begin{align*}
\overline {\mathrm{div}}(Y) =& \sum_{i=1}^{n+1}\langle \overline \nabla_{e_i}(\lambda^k\partial_r), e_i\rangle \\
= & \sum_{i=1}^{n+1}\left(e_i(\lambda^{k-1})\langle V,e_i\rangle +\lambda^{k-1}\langle \overline \nabla_{e_i}V,e_i\rangle \right)\\
=&\partial_r(\lambda^{k-1})\lambda+n\lambda^{k-1}{\lambda}'\\
=&(n+k)\lambda^{k-1}{\lambda}'.
\end{align*}
Integrate the above equation and using the divergence theorem, we have
\begin{align*}
(n+k)\int_\Omega \lambda^{k-1}{\lambda}'dv=& \int_\Omega \overline {\mathrm{div}}(Y) dv\\
= & \int_\Sigma \langle Y,\nu\rangle d\mu +\int_{\Gamma}\langle Y,\eta\rangle d\mu\\
=&\int_\Sigma \lambda^k\langle \partial_r,\nu\rangle d\mu-\lambda^k(a)|\Gamma|\\
\leq & \int_\Sigma \lambda^k d\mu-\lambda^k(a)|\Gamma|.
\end{align*}
The equality holds if and only if $\langle \partial_r,\nu\rangle=1$ everywhere on $\Sigma$, which implies that $\Sigma=\{r\}\times N^n$ is a slice.
\endproof

Suppose $\Sigma$ is a smooth, closed and embedded hypersurface in $\overline{M}^{n+1}$ which is mean convex and star-shaped in the sense that $\Sigma$ can be written as a graph over $N^n$,
\begin{equation*}
\Sigma=\mathrm{graph}~u_0,
\end{equation*}
where $u_0(\cdot )$ is a smooth positive function on $N^n$. Then $\Sigma$ is homologous to the boundary $\{a\}\times N^n$ and there is a bounded domain $\Omega$ with $\partial\Omega=\Sigma\cup (\{a\}\times N^n)$. We evolve $\Sigma$ along the inverse mean curvature flow (IMCF)
\begin{equation}\label{IMCF}
\frac{\partial }{\partial t}X=\frac{1}{H}\nu,
\end{equation}
where $\nu$ is the outward pointing unit normal vector field of the evolving hypersurface $\Sigma_t=X(\Sigma,t)$ and $H$ denotes the mean curvature of $\Sigma_t$.

Consider the quantity
\begin{equation*}
  Q(\Sigma)=|\Sigma|^{-\frac{n+k}{n}}\left(\int_\Sigma \lambda^kd\mu-k\int_\Omega \lambda^{k-1}{\lambda}'dv-\frac{k}{n+k}\lambda^k(a)|\Gamma|\right),
\end{equation*}
where $\Gamma=\{a\}\times N^n$. We have the following monotonicity for $Q(t):=Q(\Sigma_t)$ along IMCF.

\begin{proposition}\label{s2.prop}
Assume that warping function $\lambda(r)$ of $\overline{M}^{n+1}$ satisfies $\lambda'(r)>0$. Let $\Sigma_t$ be a star-shaped and mean convex solution of the IMCF.  Then  $Q'(t)\leq 0$ and $Q'(t)=0$ if and only if $\Sigma$ is a slice.
\end{proposition}
\proof
Firstly, along IMCF the area form $d\mu_t$ evolves by
\begin{equation*}
\frac{\partial}{\partial t}d\mu_t=d\mu_t.
\end{equation*}
In particular, the area evolves by
\begin{equation*}
\frac{d}{dt}|\Sigma_t|=|\Sigma_t|.
\end{equation*}
Then
\begin{align*}
  \frac{d}{dt} \int_{\Sigma_t} \lambda^kd\mu_t=&  \int_{\Sigma_t} k\lambda^{k-1}{\lambda}'\langle Dr, \frac{\partial X}{\partial t}\rangle d\mu_t+\int_{\Sigma_t} \lambda^kd\mu_t\\
=& \int_{\Sigma_t} k\lambda^{k-1}{\lambda}'\langle Dr, \frac{\nu}{H}\rangle d\mu_t+\int_{\Sigma_t} \lambda^kd\mu_t\\
\leq & \int_{\Sigma_t} k\lambda^{k-1}\frac{{\lambda}'}H d\mu_t+\int_{\Sigma_t} \lambda^kd\mu_t.
\end{align*}
Along IMCF, the co-area formula implies
\begin{equation*}
\frac{d}{dt} \int_{\Omega_t} \lambda^{k-1}{\lambda}'dv=\int_{\Sigma_t} k\lambda^{k-1}\frac{{\lambda}'}Hd\mu_t.
\end{equation*}
It follows that
\begin{align*}
&\frac{d}{dt} \left(\int_{\Sigma_t} \lambda^kd\mu_t-k\int_{\Omega_t} \lambda^{k-1}{\lambda}'dv-\frac{k}{n+k}\lambda^k(a)|\Gamma|\right)\\
\leq & \int_{\Sigma_t} \lambda^k d\mu_t\\
= & \frac{n+k}{n}\int_{\Sigma_t} \lambda^kd\mu_t-\frac{k}{n}\int_{\Sigma_t} \lambda^kd\mu_t\\
\leq &\frac{n+k}{n}\int_{\Sigma_t} \lambda^kd\mu_t-\frac{k(n+k)}{n}\int_{\Omega_t} \lambda^{k-1}{\lambda}'dv-\frac{k}{n}\lambda^k(a)|\Gamma|\\
=&\frac{n+k}{n}\left(\int_{\Sigma_t} \lambda^kd\mu_t-k\int_{\Omega_t} \lambda^{k-1}{\lambda}'dv-\frac{k}{n+k}\lambda^k(a)|\Gamma|\right),
\end{align*}
where we used \eqref{s2.lemma1} in the second inequality. We conclude that $Q'(t)\leq 0$. The equality holds if and only if $\langle \partial_r,\nu\rangle=1$ everywhere on $\Sigma_t$, which is equivalent to that $\Sigma_t$ is a slice.
\endproof

\section{Inequalities involving $k$-th boundary momentum, area and weighted volume}\label{sec ineq k boundary moment}

In this section, we apply the monotonicity of $Q(t)$ in Proposition \ref{s2.prop} and the convergence results of IMCF to derive the weighted geometric inequalities of the form \eqref{ineq eg1}.

\subsection{Inequalities in the space forms}
We first look at the space forms case.
\begin{theorem}
Let $k\geq 1$. Suppose that either
\begin{enumerate}
\item $\Sigma$ is a smooth, star-shaped mean convex hypersurface in the Euclidean space $\mathbb{R}^{n+1}$; or
\item $\Sigma$ is a smooth, star-shaped mean convex hypersurface in the hyperbolic space $\mathbb{H}^{n+1}$; or
\item $\Sigma$ is a smooth convex hypersurface in the sphere $\mathbb{S}^{n+1}$.
\end{enumerate}
Denote $\Omega$ the bounded domain enclosed by $\Sigma$. Then
\begin{equation}\label{s2.ineqn}
\int_\Sigma \lambda^kd\mu\geq \frac{n}{n+k} \omega_{n}^{-\frac{k}{n}}|\Sigma|^{\frac{n+k}{n}}+k\int_\Omega \lambda^{k-1}{\lambda}'dv,
\end{equation}
where $\omega_n=|\mathbb{S}^n|$ is the area of the unit sphere $\mathbb{S}^n$. The equality holds if and only if $\Sigma$ is a geodesic sphere centered at the origin.
\end{theorem}
\proof
\textbf{Case 1} - in the Euclidean space. In this case, $\lambda(r)=r$ and the quantity
\begin{align*}
Q(t)=|\Sigma_t|^{-\frac{n+k}{n}}\left(\int_{\Sigma_t} r^k-\int_{\Omega_t} k r^{k-1}\right)
\end{align*}
is monotone nonincreasing along the IMCF \eqref{IMCF}. By the result of Gerhardt \cite{Ge90} and Urbas \cite{Ur90}, the flow hypersurface $\Sigma_t$ of IMCF remains to be star-shaped, mean convex, and expands to infinity. The rescaled IMCF converges to a sphere $\widetilde \Sigma_\infty=r_\infty \mathbb{S}^n$. Since $Q(t)$ is scale-invariant, it suffices to show that the inequality holds on a sphere. Let $\widetilde \Omega_\infty$ be the ball bounded by $\widetilde \Sigma_\infty$, then  \eqref{s2.lemma1} shows that
\begin{equation}\label{s3.1-1}
  \int_{\widetilde \Sigma_\infty} r^k - \int_{\widetilde \Omega_\infty}k r^{k-1}\ge \frac{n}{n+k}\int_{\widetilde \Sigma_\infty} r^{k} =\frac{n}{n+k}\omega_nr_\infty^{n+k}.
\end{equation}
Therefore $Q(\widetilde \Sigma_\infty)$ is bounded below by $\frac{n}{n+k}  \omega_{n}^{-\frac{k}{n}} $. From this the inequality \eqref{s2.ineqn} follows.

\textbf{Case 2} - in the hyperbolic space. Gerhardt \cite{Ge11} proved that the solution $\Sigma_t$ of the IMCF \eqref{IMCF} remains to be star-shaped and mean convex, $\Sigma_t$ expands to the infinity and the principal curvatures $\kappa_i$ decays to $1$ exponentially as $t\to\infty$. Let $g_{ij}$ be the induced metric on $\Sigma_t$. The asymptotical behavior of $\Sigma_t$ along IMCF proved by Gerhardt \cite{Ge11} implies that
\begin{align*}
\sqrt{\det g} =& \lambda^{n} \sqrt{\det g_{\mathbb{S}^{n}}}\left(1+O(e^{-\frac{2}{n}t})\right) \\
\lambda(r) =& O(e^{\frac{t}{n}}).
\end{align*}
Then by the inequality \eqref{s2.lemma1}, we have
\begin{align*}
& \int_{\Sigma_t} \lambda^kd\mu_t- k\int_{\Omega_t} \lambda^{k-1}{\lambda}'dv \\
\geq & \frac{n}{n+k}\int_{\Sigma_t} \lambda^kd\mu_t\\
=&\frac{n}{n+k}\int_{\mathbb{S}^{n}}\lambda^{n+k}(r)d\mu_{\mathbb{S}^{n}}\left(1+O(e^{-\frac{2}{n}t})\right).
\end{align*}
On the other hand, the H\"{o}lder inequality implies that
\begin{align*}
  |\Sigma_t|^{\frac{n+k}{n}} =& \biggl(\int_{\mathbb{S}^{n}}\lambda^{n}(r)d\mu_{\mathbb{S}^{n}} \left(1+O(e^{-\frac{2}{n}t})\right)\biggr)^{\frac{n+k}{n}}\\
\leq & \omega_{n}^{\frac{k}{n}}\int_{\mathbb{S}^{n}}\lambda^{n+k}(r)d\mu_{\mathbb{S}^{n}}\left(1+O(e^{-\frac{2}{n}t})\right).
\end{align*}
It follows that
\begin{align*}
  \lim_{t\to \infty}Q(t) =& \lim_{t\to \infty} |\Sigma_t|^{-\frac{n+k}{n}}\left(\int_{\Sigma_t} \lambda^kd\mu_t-k\int_{\Omega_t} \lambda^{k-1}{\lambda}'dv\right)\\
\geq & \frac{n}{n+k}\omega_{n}^{-\frac{k}{n}}.
\end{align*}
This implies the inequality \eqref{s2.ineqn}.

\textbf{Case 3} - in the sphere.  In this case, $\lambda(r)=\sin r$. Gerhardt \cite{Ge15} (see also \cite{MS15}) proved that for smooth strictly convex hypersurface $\Sigma$ in the sphere $\mathbb{S}^{n+1}$, the solution $\Sigma_t$ of the IMCF \eqref{IMCF} remains to be strictly convex and converges to the equator smoothly as $t\to T^*<\infty$, and properly rescaled solution converges to a sphere. We have
\begin{align*}
\lim_{t\to T^*} |\Sigma_t|=& \omega_{n}, \\
\lim_{t\to T^*} \int_{\Sigma_t} \lambda^k(r)d\mu_t= & \omega_{n},\\
\lim_{t\to T^*}\int_{\Omega_t}\lambda^{k-1}{\lambda}'dv=&\frac{\omega_{n}}{n+k}.
\end{align*}
Then
\begin{equation*}
\lim_{t\to T^*}Q(t)=\frac{n}{n+k}\omega_{n}^{-\frac{k}{n}}.
\end{equation*}
By the monotonicity of $Q(t)$, we conclude that
\begin{equation*}
Q(0)\geq \lim_{t\to T^*}Q(t)=\frac{n}{n+k}\omega_{n}^{-\frac{k}{n}}.
\end{equation*}
This is equivalent to \eqref{s2.ineqn}.

In all cases, the equality of \eqref{s2.ineqn} implies that $Q(0)=Q(t)$ for all $t>0$. In particular, $Q'(0)=0$ and so $\Sigma$ is a slice.
\endproof

\subsection{Inequalities in warped product manifolds}
In this subsection, we show that the aforementioned inequality remains valid in other warped product manifolds, subject to certain conditions on the warping functions.

Let $\overline{M}^{n+1}=[a,\infty)\times N^n$ be a warped product manifold with the metric
\begin{equation*}
\overline{g}=dr^2+\lambda^2(r)g_N,
\end{equation*}
where $\lambda(r)$ is a smooth positive function on $(a,\infty)$ with $\lambda''(r)\geq 0$ and $\lambda'(r)>0$. We assume further that the warping function $\lambda(r)$ satisfies
\begin{equation}\label{s3.assm2}
  \limsup_{r\to\infty}\frac{\lambda''\lambda}{\lambda'^2}<\infty,\quad and\quad \limsup_{r\to\infty,\\\lambda''(r)>0}\frac{\lambda'''\lambda}{\lambda'\lambda''}<\infty.
\end{equation}

\begin{theorem}\label{scheuer}
Let $k\geq 1$, and $\Sigma$ be a smooth, closed, mean convex and star-shaped hypersurface in the warped product manifold $\overline{M}^{n+1}=[a,\infty)\times N^n$ with warping function $\lambda(r)$ satisfying $\lambda''(r)\geq 0$, $\lambda'(r)>0$ and \eqref{s3.assm2}. Assume that $(N^n,g_N)$ has non-negative sectional curvature and both of the following two items hold:
\begin{enumerate}
\item In case that $\sup_{r>a}\lambda'(r)<\infty$, we assume further that $\mathrm{Ric}_N>0$.
\item In case that $\sup_{r>a}\lambda'(r)=\infty$, we assume that $\liminf\limits_{r\to\infty}\lambda''\lambda/\lambda'^2>0$.
\end{enumerate}
Denote $\Omega$ the domain enclosed by $\Sigma$ and $\Gamma=\{a\}\times N^n$. Then
\begin{equation}\label{s3.ineqn}
  \int_\Sigma \lambda^kd\mu\geq  \frac{n}{n+k} |N|^{-\frac{k}{n}}|\Sigma|^{\frac{n+k}{n}}+k\int_\Omega \lambda^{k-1}{\lambda}'dv+\frac{k}{n+k}\lambda^k(a)|\Gamma|.
\end{equation}
The equality holds if and only if $\Sigma$ is a slice $\{r\}\times N^n$.
\end{theorem}
\proof
We consider the IMCF \eqref{IMCF} with $\Sigma$ as the initial hypersurface.  Scheuer \cite[Theorem 1.3]{Sch19} proved that the solution of \eqref{IMCF} exists for all time $t\in [0,+\infty)$ and the flow hypersurfaces $\Sigma_t$ become umbilical with the rate
\begin{equation*}
\left|h_i^j-\frac{\lambda'}{\lambda}\delta_i^j\right|\leq \frac{ct}{\lambda'\lambda},\quad \forall~t\in [0,\infty),
\end{equation*}
where $c$ is a constant depending only on the initial hypersurface $\Sigma$ and the $t$-factor can be replaced by $e^{-\alpha t}$ when $\lambda'$ is bounded.  The flow hypersurfaces $\Sigma_t$ are graphs over $N^n$
\begin{equation*}
\Sigma_t=\mathrm{graph}~u(t,\cdot)
\end{equation*}
with $u\in C^\infty([0,+\infty)\times N,(a,\infty))$. The graph function $u(t,\cdot)$ tends to $\infty$ as $t\to\infty $ and its gradient satisfies the estimate (\cite[Lemma 4.8]{Sch19})
\begin{equation}\label{est-Du}
\frac{|Du|^2}{\lambda^2}\leq ce^{-\alpha t},\quad \forall~t\in [0,\infty)
\end{equation}
for two positive constants $c,\alpha$.

Let $g_{ij}$ be the induced metric on $\Sigma_t$. Then
\begin{equation*}
g_{ij}=u_iu_j+\lambda^2\sigma_{ij}=\lambda^2(\sigma_{ij}+\lambda^{-2}u_iu_j),
\end{equation*}
where $\sigma_{ij}$ denotes the components of the metric $g_N$. The estimate \eqref{est-Du} implies that
\begin{align}\label{det g}
\sqrt{\det g} =& \lambda^{n} \sqrt{\det g_N}\left(1+O(e^{-\alpha t})\right).
\end{align}
Then by the inequality \eqref{s2.lemma1}, we have
\begin{align*}
& \int_{\Sigma_t} \lambda^kd\mu_t- k\int_{\Omega_t} \lambda^{k-1}{\lambda}'dv -\frac{k}{n+k}\lambda^k(a)|\Gamma|\\
\geq & \frac{n}{n+k}\int_{\Sigma_t} \lambda^kd\mu_t\\
=&\frac{n}{n+k}\int_{N^{n}}\lambda^{n+k}(r)d\mu_{N}\left(1+O(e^{-\alpha t})\right).
\end{align*}
On the other hand, the H\"{o}lder inequality implies that
\begin{align*}
|\Sigma_t|^{\frac{n+k}{n}} =& \biggl(\int_{N^{n}}\lambda^{n}(r)d\mu_{N} \left(1+O(e^{-\alpha t})\right)\biggr)^{\frac{n+k}{n}}\\
\leq & |N|^{\frac{k}{n}}\int_{N^{n}}\lambda^{n+k}(r)d\mu_{N}\left(1+O(e^{-\alpha t})\right),
\end{align*}
It follows that
\begin{align*}
  \lim_{t\to \infty}Q(t) =& \lim_{t\to \infty} |\Sigma_t|^{-\frac{n+k}{n}}\left(\int_{\Sigma_t} \lambda^kd\mu_t-k\int_{\Omega_t} \lambda^{k-1}{\lambda}'dv-\frac{k}{n+k}\lambda^k(a)|\Gamma|\right)\\
\geq & \frac{n}{n+k}|N|^{-\frac{k}{n}}.
\end{align*}
This implies the inequality \eqref{s3.ineqn}. If equality holds in \eqref{s3.ineqn}, then $Q'(t)=0$ for all $t>0$. This implies that $\Sigma$ is a slice $\{r\}\times N^n$ for some $r>a$.
\endproof

A similar result can be obtained when the fiber $N^n$ of the warped product $\overline{M}^{n+1}=[a, \infty) \times N$ has positive Ricci curvature. The restrictions on the warping function $\lambda$ are different in this case, specifically with the requirement that $\lambda'$ is bounded.
\begin{theorem}
Assume that $k \ge 1$ and let $\Sigma$ denote a smooth, closed, mean convex and star-shaped hypersurface in the warped product $\overline{M}^{n+1}=[a, \infty) \times N$, equipped with the metric $\overline{g}=dr^2+\lambda^2(r)g_N$, where $(N, g_N)$ is a closed Riemannian manifold with positive Ricci curvature.  Assume that $\lambda$ satisfies
$$C \ge \lambda^{\prime}(r) > 0, \quad C \ge \lambda^{1+\alpha}(r) \lambda^{\prime \prime}(r) \ge 0$$
for some positive constants $C$ and $\alpha$.
Denote $\Omega$ the domain enclosed by $\Sigma$ and $\Gamma=\{a\} \times N$. Then
$$
\int_{\Sigma} \lambda^k d \mu \ge \frac{n}{n+k} |N|^{-\frac{k}{n}}|\Sigma|^{\frac{n+k}{n}}+k \int_{\Omega} \lambda^{k-1} \lambda^{\prime} d v+\frac{k}{n+k} \lambda^k(a)|\Gamma|.
$$
The equality holds if and only if $\Sigma$ is a slice.
\end{theorem}
\begin{proof}
The proof is almost identical to Theorem \ref{scheuer}. In this case, the asymptotic \eqref{det g} is implied by Theorem 5.1 in H. Zhou's work \cite{zhou2018inverse} and the assumption that $\lambda'$ is uniformly bounded as $r\to\infty$. The rest of the proof is the same.
\end{proof}

\section{Inequalities involving weighted curvature integrals and quermassintegrals}\label{sec quermassintegrals}

Let $\overline{M}^{n+1}(K)$ be the simply connected space form of constant sectional curvature $K=-1, 0, 1$, which is viewed as a warped product manifold as described in \eqref{s3.spaceforms}.  We define
\begin{equation}\label{s4.Phi}
\Phi(r)=\int_0^r \lambda(s) d s=\left\{\begin{array}{cl}
\cosh r-1, & K=-1, \\
\dfrac{r^2}{2}, & K=0, \\
1-\cos r, & K=1.
\end{array}\right.
\end{equation}
It is well known that the vector field $V=\overline{\nabla} \Phi=\lambda(r) \partial_r$ on $\overline{M}^{n+1}(K)$ satisfies $\overline{\nabla}\left(\lambda(r) \partial_r\right)=\lambda^{\prime}(r) \overline{g}$, and hence is a conformal Killing field.

For a convex bounded domain $\Omega$ in the space form $\overline{M}^{n+1}(K)$, the $k$-th quermassintegral $W_k$ of $\Omega$ is defined as the measure of the set of totally geodesic $k$-dimensional subspaces which intersect $\Omega$ (\cite{San04}, \cite[Definition 2.1]{So05}) \footnote{The definition \eqref{s5.Wk} differs from Definition 2.1 in \cite{So05} by a constant multiple $\frac{n+1-k}{n+1}$.}:
\begin{equation}\label{s5.Wk}
W_k(\Omega)=\frac{\omega_{k-1} \cdots \omega_0}{\omega_{n-1} \cdots \omega_{n-k}} \int_{\mathcal{L}_k} \chi\left(L_k \cap \Omega\right) d L_k, \quad k=1, \ldots, n
\end{equation}
where $\mathcal{L}_k$ is the space of $k$-dimensional totally geodesic subspaces in $\overline{M}^{n+1}(K)$, $\omega_k=\left|\mathbb{S}^k\right|$ is the area of the unit round sphere $\mathbb{S}^k$, and the function $\chi$ is defined to be 1 if $L_k \cap \Omega \neq \emptyset$ and to be 0 otherwise. Furthermore, we have
$$
W_0(\Omega)=|\Omega|, \quad W_1(\Omega)=\frac{1}{n}|\partial \Omega|, \quad W_{n+1}(\Omega)=\left|\mathbb{B}^{n+1}(1)\right|=\frac{\omega_n}{n+1}.
$$
Assume that $\Sigma=\partial \Omega$ is smooth. Then the quermassintegrals are related to the curvature integrals of $\Sigma$ by (\cite[Proposition 7]{So05})
$$
\begin{aligned}
\int_\Sigma E_k(\kappa) d \mu & =(n-k) W_{k+1}(\Omega)-kKW_{k-1}(\Omega), \quad k=1, \ldots, n-1, \\
\int_\Sigma E_n(\kappa) d \mu & =\omega_n-nKW_{n-1}(\Omega),
\end{aligned}
$$
where $E_k(\kappa)$ is the normalized $k$-th mean curvature of $\Sigma$ defined as
$$
E_k(\kappa)=\binom{n}{k}^{-1} \sum_{1 \le i_1<\cdots<i_k \le n} \kappa_{i_1} \cdots \kappa_{i_k},
$$
where $\kappa=\left(\kappa_1, \ldots, \kappa_n\right)$ are the principal curvatures of $\Sigma$. The quermassintegrals satisfy the following  variational property (see \cite{BC97,Wx14}):
\begin{equation}\label{s4.Wkevl}
\frac{d}{d t} W_k\left(\Omega_t\right)=\int_{\partial \Omega_t} f E_k(\kappa) d \mu_t, \quad k=0, \ldots, n
\end{equation}
along any normal variation
\begin{equation}\label{s4.flow}
\frac{\partial}{\partial t}X=f \nu,
\end{equation}
where $\nu$ denotes the unit outer normal of $\partial \Omega_t$ and $f$ is a smooth function on $\partial \Omega_t$.

In the following, we will apply inverse curvature type flows to establish geometric inequalities for domains in space forms involving three geometric quantities: the weighted curvature integral $\int_{\partial\Omega}\Phi E_k d\mu$ and two distinct quermassintegrals $W_{k-1}(\Omega)$ and $W_k(\Omega)$. This on one hand generalizes the inequalities \cite{WZ} obtained earlier by the second named author with T. Zhou involving two quantities $\int_{\partial\Omega}\Phi E_k d\mu$ and $W_k(\Omega)$, and on the other hand extends the authors' work \cite{kwong2022inverse} with G. Wheeler and V.-M. Wheeler for the two dimensional case.

Let $\Sigma$ be a smooth closed hypersurface in the space form $\overline{M}^{n+1}(K)$ with induced metric $g_{ij}$. We define the $k$-th Newton tensor $T_{k}$ associated with the shape operator
$(h_i^j)$ on $\Sigma$ as follows:
$$
\left(T_{k}\right)_{j}^{i}=\frac{1}{k !} \sum_{i_{1}, \ldots, i_{k} \atop j_{1}, \ldots, j_{k}} \epsilon_{j j_{1} \cdots j_{k}}^{i i_{1} \cdots i_{k}} h_{i_{1}}^{ j_{1}} \cdots h_{i_{k}}^{j_{k}}.
$$
The next lemma is well known (see, e.g. \cite{kwong2016extension}).
\begin{lemma}\label{lem1}
Let $\Sigma$ be a smooth closed hypersurface with induced metric in the space form $\overline{M}^{n+1}(K)$. Let $V=\overline{\nabla}\Phi$ be the vector field as above and $\nu$ be the outward unit normal of $\Sigma$. Define $u=\langle V,\nu\rangle$ as the support function of $\Sigma$. We have
\begin{equation}\label{s5.lem1}
\mathrm{div}\left(T_{k-1}(\nabla \Phi)\right)
=k\binom{n}{k}\left(\lambda^{\prime} E_{k-1}-E_{k} u\right)
\end{equation}
for $k=1,\cdots,n$. As a consequence, we have the following Minkowski formula
\begin{align}\label{s5.Mink}
\int_{\Sigma} \lambda' E_{k-1}(\kappa)d\mu = \int_{\Sigma} uE_{k}(\kappa)d\mu.
\end{align}
\end{lemma}
\begin{proof}
By Proposition 3.1 and Lemma 2.1 in \cite{kwong2016extension},
\begin{equation*}
\begin{split}
\mathrm{div}\left(T_{k-1}(\nabla \Phi)\right)
=&2 \cdot \frac{1}{2} \lambda^{\prime} \mathrm{tr}\left(T_{k-1}\right)-\left\langle T_{k-1}, A\right\rangle u\\
=&\lambda^{\prime}(n-k+1)\binom{n }{k-1} E_{k-1}-k \binom{n}{k} E_{k}u\\
=&\frac{n !}{(k-1) !(n-k) !}\left(\lambda^{\prime} E_{k-1}-E_{k} u\right)\\
=&k\binom{n}{k}\left(\lambda^{\prime} E_{k-1}-E_{k} u\right).
\end{split}
\end{equation*}
The equation \eqref{s5.Mink} follows by integrating \eqref{s5.lem1} over $\Sigma$.
\end{proof}

We also need the following variational formula.
\begin{lemma}\label{s4.lem1}
Let $\Sigma_t$ be a smooth family of closed hypersurfaces in the space form $\overline{M}^{n+1}(K)$ satisfying \eqref{s4.flow}.  For all integers $k=1,\cdots n$ we have
\begin{equation}\label{s4.evl-PhiEk}
  \frac{d}{dt}\left(\int_{\Sigma_t}\Phi E_k d\mu_t+kW_{k-1}\left(\Omega_{t}\right)\right)=\int_{\Sigma_t}\biggl((k+1)uE_k+(n-k)\Phi E_{k+1}\biggr)f d\mu_t.
\end{equation}
\end{lemma}
\proof
This follows from combining Lemma 3.2 in \cite{WZ} and the evolution equation \eqref{s4.Wkevl}.
\endproof

\subsection{Inequalities in the Euclidean space}
We first consider the Euclidean space case.
\begin{theorem}\label{s4.thmRn}
Let $\Sigma$ be a smooth, closed, star-shaped and $k$-convex hypersurface in $\mathbb R^{n+1}$ ($n\geq 2$) enclosing a bounded domain $\Omega$. Then for all $k=1,\cdots,n$, there holds
\begin{equation}\label{s4.KM3}
  \int_{\Sigma}\Phi E_kd\mu+kW_{k-1}\left(\Omega\right)\geq \frac{n+2+k}{2(n+2-k)}\omega_n\left(\frac{n+1-k}{\omega_n}\right)^{\frac{n+2-k}{n+1-k}}W_k(\Omega)^{\frac{n+2-k}{n+1-k}},
\end{equation}
where $\Phi=r^2/2$ is the function defined in \eqref{s4.Phi}. Equality holds in \eqref{s4.KM3} if and only if $\Sigma$ is a coordinate sphere.
\end{theorem}
\begin{remark}
Applying the Alexandrov-Fenchel inequality \cite{GL09}
\begin{equation*}
W_k(\Omega)\geq c_{k,\ell}W_{\ell}(\Omega)^{\frac{n+1-k}{n+1-\ell}},\quad k>\ell
\end{equation*}
for star-shaped and $k$-convex hypersurfaces, we can replace the right hand side of \eqref{s4.KM3} by $W_\ell(\Omega)^{\frac{n+2-k}{n+1-\ell}}$ for $\ell<k$. However, for $\ell<k$, this inequality reduces to Theorem 1.2 in \cite{WZ}.
\end{remark}
\proof
We consider the inverse curvature flow
\begin{equation}\label{s4.flow1}
\frac {\partial X }{\partial t}=\frac{E_{k-1}}{E_{k}}\nu,\quad k=1,\cdots,n
\end{equation}
in $\mathbb{R}^{n+1}$. If the initial hypersurface in $\mathbb{R}^{n+1}$ is star-shaped and $k$-convex, Gerhardt \cite{Ge90} and Urbas \cite{Ur90} showed that the $k$-convexity is preserved, and the solution $\Sigma_t$ expands to infinity and properly rescaled solution converges to a round sphere as $t\to\infty$.

Choosing $f={E_{k-1}}/{E_{k}}$ in \eqref{s4.evl-PhiEk}, we have that along the flow \eqref{s4.flow1},
\begin{align}\label{s4.pf1}
&\frac{d}{dt}\left(\int_{\Sigma_t}\Phi E_k d\mu_t+k W_{k-1}\left(\Omega_{t}\right)\right)\nonumber\\
=&\int_{\Sigma_t}\left((k+1)uE_{k-1}+(n-k)\Phi E_{k+1}\frac{E_{k-1}}{E_k}\right)d\mu_t\nonumber\\
\leq &\int_{\Sigma_t}\left((k+1)uE_{k-1}+(n-k)\Phi E_{k}\right)d\mu_t\nonumber\\
=&\int_{\Sigma_t}\left((k+1)E_{k-2}+(n-k)\Phi E_{k}\right)d\mu_t\nonumber\\
=&(n+2-k)(k+1)W_{k-1}\left(\Omega_{t}\right)+(n-k)\int_{\Sigma_t}\Phi E_{k}d\mu_t,
\end{align}
where we used the Newton-MacLaurin inequality
\begin{equation}\label{s4.Newt}
E_{k+1}E_{k-1}\leq E_k^2
\end{equation}
and Minkowski formula \eqref{s5.Mink}.

On the other hand, choosing $f={E_{k-1}}/{E_{k}}$ in \eqref{s4.Wkevl}, we have
\begin{align}\label{s4.pf3}
\frac{d}{dt}W_k(\Omega_t)= &\int_{\Sigma_t}E_{k-1} d\mu_t=(n+1-k) W_k(\Omega_t)
\end{align}
along the flow \eqref{s4.flow1}. Define the quantity
\begin{equation*}
Q_{k}(t)=W_k(\Omega_t)^{-\frac{n+2-k}{n+1-k}}\left(\int_{\Sigma_t}\Phi E_k d\mu_t+k W_{k-1}\left(\Omega_{t}\right)\right).
\end{equation*}
Combining \eqref{s4.pf1} and \eqref{s4.pf3} yields
\begin{align}\label{s4.pf5}
    \frac{d}{dt}Q_{k}(t)=&W_k(\Omega_t)^{-\frac{n+2-k}{n+1-k}}\biggl( \frac{d}{dt}\left(\int_{\Sigma_t}\Phi E_k d\mu_t+k W_{k-1}\left(\Omega_{t}\right)\right)\nonumber\\
   &\quad  -\frac{n+2-k}{n+1-k}\left(\int_{\Sigma_t}\Phi E_k d\mu_t +k W_{k-1}\left(\Omega_{t}\right)\right)W_k(\Omega_t)^{-1}\frac{d}{dt}W_k(\Omega_t)\biggr)\nonumber\\
\leq &W_k(\Omega_t)^{-\frac{n+2-k}{n+1-k}}\biggl((n+2-k) (k+1)W_{k-1}(\Omega_t)+(n-k)\int_{\Sigma_t}\Phi E_k d\mu_t\nonumber\\
&\quad -(n+2-k)\left(\int_{\Sigma_t}\Phi E_k d\mu_t+k W_{k-1}\left(\Omega_{t}\right)\right)\biggr)\nonumber\\
=& W_k(\Omega_t)^{-\frac{n+2-k}{n+1-k}}\biggl( (n+2-k)W_{k-1}(\Omega_t)-2\int_{\Sigma_t}\Phi E_k d\mu_t\biggr).
\end{align}
Then it follows from the inequality
\begin{equation}\label{s4.KM2}
\int_{\Sigma}\Phi(r) E_k(\kappa)d\mu~\geq~\frac {n+2-k}2W_{k-1}(\Omega),\qquad k=1,\cdots,n
\end{equation}
of Kwong and Miao \cite{kwong2014new,KM15} for $k$-convex hypersurface that the right hand side of \eqref{s4.pf5} is non-positive. Therefore,
\begin{equation}\label{s4.pf6}
\frac{d}{dt}Q_{k}(t)\leq 0
\end{equation}
along the flow \eqref{s4.flow1}. If equality holds in \eqref{s4.pf6} at some time $t$, then equality holds in \eqref{s4.Newt} everywhere on $\Sigma_t$ and equality also holds in Kwong and Miao's inequality \eqref{s4.KM2}. This implies that $\Sigma_t$ must be a coordinate sphere.

Note that the quantity $Q_{k}(t)$ is a scaling invariant. Then
\begin{align*}
 Q_{k}(0)\geq & \lim_{t\to\infty}  Q_{k}(t)=Q_{k}(\mathbb{S}^n(1))=\frac{n+2+k}{2(n+2-k)}\omega_n\left(\frac{n+1-k}{\omega_n}\right)^{\frac{n+2-k}{n+1-}}.
\end{align*}
Consequently, we obtain the inequality \eqref{s4.KM3} in Theorem \ref{s4.thmRn} for smooth, star-shaped and $k$-convex hypersurface $\Sigma$ in $\mathbb{R}^{n+1}$. Moreover, if equality holds in \eqref{s4.KM3} on $\Sigma$, then $\frac{d}{dt}Q_{k}(t)=0$ holds on the solution $\Sigma_t$ of the flow \eqref{s4.flow1} for all time $t$ which means  that the initial hypersurface $\Sigma$ is a coordinate sphere.
\endproof

\subsection{Inequalities in the hyperbolic space}
In this subsection, we prove a similar inequality for bounded static convex domain $\Omega$ in the hyperbolic space $\mathbb{H}^{n+1}$. Here we say that $\Omega$ is static convex if  the second fundamental form of the boundary $\Sigma=\partial\Omega$ satisfies
\begin{equation*}
h_{ij}>\frac{u}{\lambda'}g_{ij}>0
\end{equation*}
everywhere on $\Sigma$. This definition was firstly introduced by Brendle and Wang \cite{BW14}, which implies strict convexity but is weaker than h-convexity since $u=\langle \lambda \partial_r,\nu\rangle <\lambda'$.

The following inverse curvature type flow
\begin{equation}\label{s4.flow2}
\frac {\partial X }{\partial t}=\left(\frac{E_{k-1}}{E_{k}} -\frac{u}{\lambda^{\prime}(r)}\right)\nu
\end{equation}
was studied by Scheuer and Xia \cite{SX19}. For initially star-shaped and $k$-convex hypersurface $\Sigma$ in $\mathbb{H}^{n+1}$, they proved that the solution $\Sigma_t$ of the flow \eqref{s4.flow2} remains star-shaped and $k$-convex, $\Sigma_t$ exists for all time and converges smoothly to a geodesic sphere centered at the origin as $t\to\infty$. Later, Hu and Li \cite{HL22} showed that if the initial hypersurface $M$ is static-convex, then the solution $\Sigma_t$ is static-convex for all $t>0$. As an application, the following inequality for static convex domain $\Omega\subset \mathbb{H}^{n+1}$ was proved in \cite{HL22}:
\begin{equation}\label{ineq-HL}
\int_{\partial\Omega}\lambda'E_k d\mu\geq h_{k}\circ \chi_\ell^{-1} (W_\ell (\Omega)),
\end{equation}
for $k=0,1,\cdots,n$, $\ell=0,1,\cdots,k$, where $h_k(r):=\int_{\partial B(r)}\lambda'E_k d\mu$ and $\chi_\ell(r):=W_\ell(B(r))$.

In the following, we prove the following sharp inequality for static convex domain.
\begin{theorem}\label{s4.thmHn}
If $\Omega$ is a smooth bounded static convex domain in $\mathbb{H}^{n+1} (n\geq 2)$, then
\begin{equation}\label{ineq2}
\int_{\partial \Omega} \Phi E_{k} d \mu+kW_{k-1}(\Omega) \geq \left(\xi_k+k\chi_{k-1}\right)\circ\chi_{\ell}^{-1}\left(W_{\ell}(\Omega)\right)
\end{equation}
for all $k=1, \cdots, n$ and $\ell=0,1,\cdots,k$, where $\Phi(r)=\cosh r-1$ as defined in \eqref{s4.Phi}. The equality holds if and only if $\Omega$ is a geodesic ball centered at the origin. Here $\xi_k:[0,\infty)\to\mathbb R_+$ is a function defined by the weighted curvature integral $\xi_k(r):=\int_{\partial B(r)} \Phi E_{k} d \mu$ on the geodesic ball $B(r)$ of radius $r$, and $\chi_\ell: [0,\infty)\to \mathbb R_+$ is defined by $\chi_\ell(r):=W_\ell(B(r))$.
\end{theorem}
\proof
Along the flow \eqref{s4.flow2}, we have
\begin{align*}
  \frac{d}{dt}\biggl(\int_{\partial \Omega_t} \Phi E_{k} d \mu_t+kW_{k-1}(\Omega_t) \biggr)=& \int_{\partial\Omega_t}\left((k+1)uE_k+(n-k)\Phi E_{k+1}\right)\times \left(\frac{E_{k-1}}{E_{k}} -\frac{u}{\lambda^{\prime}(r)}\right)\\
= & (k+1)\int_{\partial\Omega_t}\frac{u}{\lambda'}\left(\lambda'E_{k-1}-uE_k\right)\\
&\quad +(n-k)\int_{\partial\Omega_t}\frac{\Phi}{\lambda'}\left(\lambda'\frac{E_{k+1}E_{k-1}}{E_k}-uE_{k+1}\right)\\
\leq & (k+1)\int_{\partial\Omega_t}\frac{u}{\lambda'}\left(\lambda'E_{k-1}-uE_k\right)\\
&\quad +(n-k)\int_{\partial\Omega_t}\frac{\Phi}{\lambda'}\left(\lambda'{E_k}-uE_{k+1}\right)\\
= &\frac{k+1}{k}\binom{n}{k}^{-1}\int_{\partial\Omega_t}\frac{u}{\lambda'}\mathrm{div}(T_{k-1}\nabla\Phi)\\
&\quad +\frac{n-k}{k+1}\binom{n}{k+1}^{-1}\int_{\partial\Omega_t}\frac{\Phi}{\lambda'}\mathrm{div}(T_{k}\nabla\Phi)\\
=&-\frac{k+1}{k}\binom{n}{k}^{-1}\int_{\partial\Omega_t}\left\langle T_{k-1}(\nabla\Phi),\nabla\left(\frac{u}{\lambda'}\right)\right\rangle \\
 &\quad -\frac{n-k}{k+1}\binom{n}{k+1}^{-1}\int_{\partial\Omega_t}\left\langle T_{k}(\nabla\Phi),\nabla\left(\frac{\Phi}{\lambda'}\right)\right\rangle,
\end{align*}
where we used Newton-MacLaurin inequality \eqref{s4.Newt}, Lemma \ref{lem1} and integration by parts. To estimate the sign of the right hand side, we note that
\begin{align*}
  \nabla_i\left(\frac{u}{\lambda'}\right) =& \frac{\nabla_iu}{\lambda'}-\frac{u}{(\lambda')^2}\nabla_i\lambda'=\frac{\nabla_i\Phi}{(\lambda')^2}(\kappa_i\lambda'-u), \\
   \nabla\left(\frac{\Phi}{\lambda'}\right)=& \frac{\nabla_i\Phi}{\lambda'}-\frac{\Phi}{(\lambda')^2}\nabla_i\lambda'=\frac{\nabla_i\Phi}{(\lambda')^2}.
\end{align*}
Since $\partial\Omega_t$ is static-convex, we have $\kappa_i\lambda'-u>0$ and both $T_{k-1}, T_k$ are positively definite. It follows that
\begin{equation*}
\frac{d}{dt}\biggl(\int_{\partial \Omega_t} \Phi E_{k} d \mu_t+kW_{k-1}(\Omega_t) \biggr)\leq 0
\end{equation*}
along the flow \eqref{s4.flow2}.

On the other hand, for $\ell\leq k$,
\begin{align*}
\frac{d}{dt}W_\ell(\Omega_t) =& \int_{\partial \Omega_t}E_\ell \left(\frac{E_{k-1}}{E_{k}} -\frac{u}{\lambda^{\prime}(r)}\right) \\
\geq &\int_{\partial \Omega_t}\frac{1}{\lambda^{\prime}}\left(\lambda' E_{\ell-1}-E_\ell u\right)\\
=&\int_{\partial \Omega_t}\frac{1}{\ell\lambda^{\prime}}\binom{n}{\ell}^{-1}\mathrm{div}(T_{\ell-1}\nabla\Phi)\\
=&\int_{\partial \Omega_t}\frac{1}{\ell(\lambda^{\prime})^2}\binom{n}{\ell}^{-1}\langle T_{\ell-1}(\nabla\Phi), \nabla \Phi\rangle ~\geq 0,
\end{align*}
where we used Newton-MacLaurin inequalities $E_\ell E_{k-1}\geq E_{\ell-1}E_k$ for $\ell\leq k$, Lemma \ref{lem1} and integration by parts.

Now we apply the above monotone properties and the convergence of the flow \eqref{s4.flow2}. Denote the limit coordinate sphere by $S_{r_\infty}=\partial\Omega_\infty$. Then
\begin{align*}
  \int_{\Sigma}\Phi E_k d\mu +kW_{k-1}(\Omega_t)\geq & \lim_{t\to\infty}\biggl(\int_{\partial \Omega_t} \Phi E_{k} d \mu_t+kW_{k-1}(\Omega_t) \biggr) \\
=&\int_{S_{r_\infty}}\Phi E_k d\mu_\infty+kW_{k-1}(\Omega_\infty)\\
=&\xi_k(r_\infty)+k\chi_{k-1}(r_\infty),\\
W_{\ell}(\Omega) \leq &\lim_{t\to\infty}W_\ell(\Omega_t)=W_\ell(\Omega_\infty)=\chi_\ell(r_\infty).
\end{align*}
As in \cite{WZ}, both functions $\xi_k:\mathbb{R}_+\to \mathbb{R}_+$ and $\chi_\ell:\mathbb{R}_+\to \mathbb{R}_+$ are strictly increasing. It follows that
\begin{equation}\label{s4.Hnpf1}
\int_{\Sigma}\Phi E_k d\mu +kW_{k-1}(\Omega)\geq \left(\xi_k+k\chi_{k-1}\right)\circ\chi_{\ell}^{-1}\left( W_{\ell}(\Omega)\right).
\end{equation}
The rigidity of \eqref{s4.Hnpf1} follows from the rigidity of the monotonicity. This completes the proof of Theorem \ref{s4.thmHn}.
\endproof

\begin{remark}
Since by definition,
\begin{equation*}
\int_{\partial\Omega}E_k d\mu=(n-k)W_{k+1}(\Omega)+kW_{k-1}(\Omega),
\end{equation*}
the left hand side of \eqref{ineq2} can be rewritten as
\begin{equation*}
\int_{\partial\Omega}\lambda'E_k d\mu-(n-k)W_{k+1}(\Omega).
\end{equation*}
This implies that our inequality \eqref{ineq2} is more sharp than the inequality \eqref{ineq-HL}.
\end{remark}

\subsection{Inequality in the sphere}
We can prove a similar inequality as in \eqref{ineq2} for convex hypersurfaces in the sphere $\mathbb{S}^{n+1}$ ($n\geq 2$) for the case $k=n$.
\begin{theorem}\label{s4.thmSn}
Let $\Sigma$ be a smooth, closed and strictly convex hypersurface in $\mathbb{S}^{n+1}$ ($n\geq 2$) enclosing a bounded domain $\Omega$. Then
\begin{align}\label{s4.eq3}
\int_{\Sigma} \Phi E_n(\kappa)d\mu +nW_{n-1}(\Omega)\geq \left(\xi_n+n\chi_{n-1}\right)\circ\chi_{\ell}^{-1}(W_\ell(\Omega)),\quad \ell=0,1,\cdots,n,
\end{align}
where $\Phi(r)=1-\cos r$ is the function defined in \eqref{s4.Phi}. Equality holds in \eqref{s4.eq3} if and only if $\Sigma$ is a coordinate sphere.
\end{theorem}

To prove \eqref{s4.eq3}, we consider the following flow in the sphere $\mathbb{S}^{n+1}$
\begin{equation}\label{s4.flow3}
\frac {\partial X }{\partial t}=\left(\frac{E_{k-1}}{E_{k}} \lambda^{\prime}(r)-u\right)\nu,
\end{equation}
which was introduced by Brendle, Guan and Li \cite{BGL}. If the initial hypersurface is strictly convex, it was proved in \cite{BGL} that the flow \eqref{s4.flow3} for $k=n$ in the sphere $\mathbb S^{n+1}$ preserves the strict convexity and converges smoothly to a geodesic sphere. The proof of \eqref{s4.eq3}  relies on this convergence result and the following monotone property.
\begin{lemma}\label{s4.lem3}
Assume that $\Sigma_{t}=\partial \Omega_t$ is a smooth closed and strictly convex solution to the flow \eqref{s4.flow3} in $\mathbb{S}^{n+1}$. We have that $\displaystyle \int_{\partial \Omega_{t}} \Phi E_{k}+k W_{k-1}\left(\Omega_{t}\right)$ is non-increasing in time $\displaystyle t$ and is strictly decreasing unless $\Sigma_{t}$ is a coordinate sphere.
\end{lemma}
\begin{proof}
By choosing
$$
f=\frac{E_{k-1}}{E_{k}} \lambda^{\prime}(r)-u
$$
in \eqref{s4.evl-PhiEk} and using the Newton-MacLaurin inequality \eqref{s4.Newt}, we have
\begin{align*}
\frac{d}{dt}\left(\int_{\partial \Omega_{t}} \Phi E_k+k W_{k-1}\left(\Omega_{t}\right)\right)
=&\int_{\Sigma_{t}}(k+1) u\left(E_{k-1} \lambda^{\prime}(r)-u E_{k}\right)\\
&+\int_{\Sigma_{t}}(n-k) \Phi\cdot\left(E_{k+1} \frac{E_{k-1}}{E_{k}} \lambda^{\prime}(r)-u E_{k+1}\right)\\
\leq & \int_{\Sigma_{t}}(k+1) u\left(E_{k-1} \lambda^{\prime}(r)-u E_{k}\right)\\
&+\int_{\Sigma_{t}}(n-k) \Phi\cdot\left(E_{k} \lambda^{\prime}(r)-u E_{k+1}\right).
\end{align*}
By Lemma \ref{lem1},
\begin{align*}
\frac{d}{dt}\left(\int_{\partial \Omega_{t}} \Phi E_k+k W_{k-1}\left(\Omega_{t}\right)\right)
\le & \frac{(k+1)}{k\binom{n}{k}} \int_{\partial \Omega_{t}} u \mathrm{div}\left(T_{k-1}(\nabla \Phi)\right)+ \frac{(n-k)}{(k+1)\binom{n}{k+1}} \int_{\partial \Omega_{t}} \Phi \mathrm{div}\left(T_{k}(\nabla \Phi)\right)\\
= & -\frac{(k+1)}{k\binom{n}{k}} \int_{\partial \Omega_{t}} \langle T_{k-1}(\nabla \Phi), A\nabla \Phi \rangle - \frac{(n-k)}{(k+1)\binom{n}{k+1}} \int_{\partial \Omega_{t}} \langle T_{k}(\nabla \Phi), \nabla \Phi\rangle\\
= & - \int_{\partial \Omega_{t}}\left\langle \left(\frac{(k+1)}{k\binom{n}{k}}T_{k-1}A+\frac{(n-k)}{(k+1)\binom{n}{k+1}}T_k\right)(\nabla \Phi), \nabla \Phi \right\rangle
\end{align*}
where we have used the fact $\nabla u=A\nabla \Phi$.

If $\Sigma_{t}$ is convex, the operator
$\frac{(k+1)}{k\binom{n}{k}}T_{k-1}A+\frac{(n-k)}{(k+1)\binom{n}{k+1}}T_k$ is positive. Therefore,
$$ \frac{d}{d t} \left(\int_{\partial \Omega_{t}} \Phi E_k+k W_{k-1}\left(\Omega_{t}\right)\right)\le 0. $$
The inequality is strict unless $\nabla \Phi \equiv 0$ on $\Sigma_{t}$ which means that $\Sigma_{t}$ is a geodesic sphere centered at the origin.
\end{proof}

\begin{proof}[Proof of Theorem \ref{s4.thmSn}]
By Lemma 5.1 in \cite{WZ}, the $\ell$th quermassintegral $W_{\ell}\left(\Omega_t\right)$ is increasing along the flow \eqref{s4.flow3}.  Then Theorem \ref{s4.thmSn} follows by combining the convergence result of the flow \eqref{s4.flow3} for $k=n$ and the monotone property in Lemma \ref{s4.lem3}.
\end{proof}

\bibliographystyle{amsplain}

\end{document}